\documentclass[12pt,reqno]{amsart}
\usepackage{amsfonts,color,amsmath,amssymb,fancyhdr}
\usepackage{amsthm}
\usepackage{txfonts}
\usepackage{float}
\usepackage{multirow}
\usepackage[figuresright]{rotating}
\usepackage{graphicx}
\newtheorem{thm}{Theorem}[section]
\newtheorem{lemma}[thm]{Lemma}

\newtheorem{prop}[thm]{Proposition}
\newtheorem{definition}[thm]{Definition}

\numberwithin{equation}{section}
\newtheorem{remark}[thm]{Remark}

\newcommand{\bF}{\mathbb F}

\newcommand{\bZ}{\mathbb Z}

  \def\Tr{{\rm Tr}}

\newcommand{\<}{\langle}
\renewcommand{\>}{\rangle}

\newcommand{\tup}{\textup}

\begin{document}
\title{ On the irreducible characters of Suzuki $p$-groups  }

\begin{center}
\author{Wendi Di, Tao Feng and Zhiwen He$^*$ }
\end{center}
\address{School of Mathematical Sciences, Zhejiang University, Hangzhou 310027,  China}
\email{Wendyjj@zju.edu.cn}
\address{School of Mathematical Sciences, Zhejiang University, Hangzhou 310027, China}
\email{tfeng@zju.edu.cn}

\address{School of Mathematical Sciences, Zhejiang University, Hangzhou 310027,  China}
\email{zhiwenhe94@163.com}

\begin{abstract}
In this paper, we completely determine the irreducible characters of the four families of Suzuki $p$-groups.
\end{abstract}

\keywords{Suzuki $p$-groups; $VZ$-groups; irreducible characters.\\
{\bf  Mathematics Subject Classification (2010) 20C15, 20D15, 20E45}\\
{\bf  Funding information: National Natural Science Foundation of China under Grant No. 11771392.}\\
$^*$Correspondence author}

\maketitle

\section{Introduction}\label{background}

Suzuki $2$-groups are a class of finite $2$-groups which  possess a cyclic group of automorphisms that permutes their involutions transitively. Higman \cite{Higman} divided all Suzuki $2$-groups into four classes as $A,B,C,D$ type  in $1963$.  The definitions  also make sense for $p$ odd, so we talk about Suzuki $p$-groups for a general prime $p$.
We adopt the notation  in  \cite{Suzuki} and refer to the groups as $A_p(m,\theta)$,  $B_p(m,\theta,\varepsilon)$, $C_p(m,\theta,\varepsilon)$, $D_p(m,\theta,\varepsilon)$.
The structure of  Suzuki $p$-groups has been considered in   \cite{Suzuki, Sagirov1,Sagirov2}. Here we  focus on the characters of Suzuki $p$-groups.
 The character degrees of Suzuki $p$-groups of type $A$ with $p=2$ and $p$ odd are determined in $1999$ and $2003$, separately, by Sagirov  in \cite{Sagirov1, Sagirov2} and his work can be found in chapter $46$ of \cite{Groups of prime power order}. In $2015$, Le et al. \cite{Suzuki} reproved Sagirov's  result on $A_p(m,\theta)$ and  gave the irreducible characters of $A_p(m,\theta)$ by using a different technique.  Besides, they extended their method to $C_p(m,\theta,\varepsilon)$ and determined the character degrees  of  $C_p(m,\theta,\varepsilon)$. However, the paper contains some errors, and the character degrees of the remaining two families of Suzuki $p$-groups have not been determined.

In this paper, we extend the method of Le et al. and  determine the  irreducible characters of Suzuki $p$-groups of all four families completely.
Along the way, we correct some mistakes in \cite{Suzuki} in the determination of irreducible characters of $A_2(m,\theta)$ and $C_2(m,\theta,0)$.

Studying on the characters of Suzuki $p$-groups has important significance in the branch of  combinatorics. There have been many ground-breaking results on some interesting combinatorial objects (such as difference sets, partial difference sets, strong regular graphs, etc.) in  non-abelian groups by using representation and character theory. For example,  one can refer to  \cite{Davis,Gow,Ma} for more detailed information.
We believe deeply that Suzuki $p$-groups has many combinatorial properties. Actually, we have constructed central difference sets and central partial difference sets in the Suzuki $p$-groups of type $A$ by its characters \cite{Di}.

This article is organized as follows. In Section \ref{introduction}, we introduce  the four families of Suzuki $p$-groups and list some  lemmas about characters of finite groups which will be used  later. In Section \ref{main results}, we give our main results on the four families of Suzuki $p$-groups. In Sections \ref{pf_main result}, we prove our main results. In Section \ref{Sec5} we list the explicit expressions of the linear  characters of some abelian groups. Combining them with the results in Section \ref{main results}, the expressions of irreducible characters of Suzuki $p$-group of type $A$ are clear.

At the end of this section, we introduce some notation that we will use in this paper. For a finite group $G$, let $\tup{k}(G)$ denote  the class number of $G$ and use $\tup{Irr}(G)$ for the set of irreducible characters of $G$. Furthermore, we denote  the set of linear characters of $G$, the set of irreducible characters of $G$ of degree $r$ and the set of non-linear irreducible characters of $G$ by $\tup{Lin(G)}$, $\tup{Irr}_{(r)}(G)$ and $\tup{Irr}_{1}(G)$, respectively.
Let $\tup{cd}(G)$ be the set of character degrees of $G$, that is to say, $\tup{cd}(G)=\{\chi(1): \chi\in \tup{Irr}(G)\}$.
 For any positive integer $r$, let $\bF_{p^r}$ be the finite field with $p^r$ elements and  $\tup{Tr}_r$  be the absolute trace function from $\bF_{p^r}$ to $\bF_p$. If $d$ is a divisor of $r$, then let $\tup{Tr}_{r/d}$ be the relative trace function from $\bF_{p^r}$ to $\bF_{p^d}$.
\section{Preliminaries }\label{introduction}

Suppose that $p$ is a prime and $m$ is a positive integer. Write $\bF:=\bF_{p^m}$. In Table \ref{table1}, we list the four families of Suzuki $p$-groups. In the table, $\varepsilon$ is an element of $\bF$, $\theta$ is a field automorphism of $\bF$ such that $a^\theta=a^{p^l}$ for each $a\in\bF$ and some positive integer $l$. Set
\[
n=\gcd(l,m), \; k=m/n,
\]
then $o(\theta)=k$. Let
\[
\bF_{\theta}:=\{x\in\bF: x^\theta=x\},
\]
it is the fixed subfield of $\theta$, and $\bF_\theta=\bF_{p^n}$.
\begin{table}[H]
\caption{The Suzuki $p$-groups of four types}\label{table1}
\begin{tabular}{c|c|c}

\hline
 \textup{type} & \textup{underlying set} &  \textup{multiplication}\\
 \hline
$A_p(m,\theta)$&$\bF\times\bF$&$(a,b)(c,d)=(a+c,b+d+ac^{\theta})$\\
\hline
$B_p(m,\theta,\varepsilon)$&$\bF\times\bF\times\bF$&$(a,b,c)(d,e,f)=(a+d,b+e,c+f+ad^{\theta}+be^{\theta}+\varepsilon ae^{\theta})$\\
\hline
$C_p(m,\theta,\varepsilon)$&$\bF\times\bF\times\bF$&$(a,b,c)(d,e,f)=(a+d,b+e,c+f+ad^{\theta}+be+\varepsilon a^{1/p}{e^p}^{\theta})$\\
\hline
$D_p(m,\theta,\varepsilon)$&$\bF\times\bF\times\bF$&$(a,b,c)(d,e,f)=(a+d,b+e,c+f+ad^{\theta}+be^{\theta^2}+\varepsilon a^{\theta^3}e^{\theta})$\\
\hline
\end{tabular}
\end{table}

\begin{definition}\label{f_a}
 For $\theta\in \tup{Aut}(\bF)$ and  $a\in \bF^{*}$, define
\begin{equation}\label{eq_f_a}
f_{a,\theta}:\bF\rightarrow \bF, \; x\mapsto ax^{\theta}-xa^{\theta}.
\end{equation}
\end{definition}
We observe that $f_{a,\theta}$ is $\bF_\theta$-linear and 
  $\textup{Im}(f_{a,\theta})$ is an $\bF_{\theta}$-hyperplane of $\bF$ if $\theta$ is non-trivial.

 In Table \ref{table_commutator}, we list the commutators of the four families  of Suzuki $p$-groups. The conditions in the last column is to guarantee that the group is non-abelian.
\begin{table}[H]
\caption{The commutators  of Suzuki $p$-groups}\label{table_commutator}
\begin{tabular}{c|c|c}
\hline
 \textup{type} &\textup{commutator}& remark  \\
 \hline
$A_p(m,\theta)$&$[(a,b),(c,d)]=(0,f_{a,\theta}(c))$ &$\theta\neq1$\\
\hline
$B_p(m,\theta,\varepsilon)$&$[(a,b,c),(d,e,f)]=(0,0,f_{a,\theta}(d)+f_{b,\theta}(e)+\varepsilon(ae^{\theta}-db^{\theta}))$&$\theta\neq1$ or $\varepsilon\neq0$\\
\hline
$C_p(m,\theta,\varepsilon)$&$[(a,b,c),(d,e,f)]=(0,0,f_{a,\theta}(d)+\varepsilon(a^{1/p}e^{p\theta}-d^{1/p}b^{p\theta}))$&$\theta\neq1$ or $\varepsilon\neq0$\\
\hline
$D_p(m,\theta,\varepsilon)$&$[(a,b,c),(d,e,f)]=(0,0,f_{a,\theta}(d)+f_{b,\theta^2}(e)+\varepsilon(a^{\theta^3}e^{\theta}-d^{\theta^3}b^{\theta}))$&$\theta\neq1$ or  $\varepsilon\neq0$\\
\hline
\end{tabular}
\end{table}

\begin{lemma}\label{Im(f_a)}
Let $\theta\in \tup{Aut}(\bF)$, $o(\theta)=k>1$, and write $n=m/k$.
Suppose $f_{a,\theta}$ is the map defined in \eqref{eq_f_a}. Then the following holds:
\begin{enumerate}
  \item $\tup{Im}(f_{a,\theta})=\{x\in\bF:\tup{Tr}_{\bF/\bF_{\theta}}((aa^{\theta})^{-1}x)=0\}$, where $\tup{Tr}_{\bF/\bF_{\theta}}$ is the relative trace function from $\bF$ to  $\bF_\theta$.
  \item For $a,b\in\bF^*$, $\tup{Im}(f_{a,\theta})=\tup{Im}(f_{b,\theta})$ if and only if $ab^{-1}\in \bF_{\theta^2}^*$.
\end{enumerate}
\end{lemma}
\begin{proof}
\begin{enumerate}
  \item We have
\[
    f_{a,\theta}(x)=aa^\theta((xa^{-1})^\theta-xa^{-1})=aa^\theta f_{1,\theta}(xa^{-1}),
\]
so $\tup{Im}(f_{a,\theta})=aa^\theta\tup{Im}(f_{1,\theta})$ for any $a\in \bF^*$.
Since elements of $\tup{Im}(f_{1,\theta})$ have relative trace $0$ to the subfield $\bF_{\theta}$ and  $\tup{Im}(f_{1,\theta})$  has $\bF_\theta$-dimension $k-1$, we deduce that
\[
  \tup{Im}(f_{1,\theta})=\{x\in\bF:\tup{Tr}_{\bF/\bF_{\theta}}(x)=0\}
\]
 and the first claim follows.
  \item
  It is deduced from $(1)$ that  $\tup{Im}(f_{a,\theta})=\tup{Im}(f_{b,\theta})$ occurs only when \[\{x\in\bF:\tup{Tr}_{\bF/\bF_{\theta}}((aa^{\theta})^{-1}x)=0\} \tup{ and } \{x\in\bF:\tup{Tr}_{\bF/\bF_{\theta}}((bb^{\theta})^{-1}x)=0\}\] are the same $\bF_\theta$-hyperplanes. This is the case  $aa^\theta(bb^{\theta})^{-1}\in\bF_\theta$. So we have
  \[(ab^{-1})^\theta(ab^{-1})^{\theta^2}=ab^{-1}(ab^{-1})^\theta, \,i.e.,\, (ab^{-1})^{\theta^2}=ab^{-1}.\]
\end{enumerate}

\end{proof}
Note that if $\theta$ has order $2$, we have $\bF_{\theta^2}^*=\bF^*$.  So in this case we deduce from Lemma \ref{Im(f_a)} $(2)$ that $\tup{Im}(f_{a,\theta})=\tup{Im}(f_{1,\theta})$ for all $a\in\bF^*$.
Now we describe the derived subgroups and the centers of the four families of Suzuki $p$-groups in Table \ref{table_Z(G)&G'}.

\begin{table}[H]
\caption{The centers  and the derived subgroups }\label{table_Z(G)&G'}
\begin{tabular}{cccc}
\hline
 $G$ & Conditions &$G'$ &$Z(G)$\\
 \hline
\multirow{2}{*}{$A_p(m,\theta)$}& $k=2$& $\{(0,c):c\in\tup{Im}(f_{1,\theta})\}$ &\multirow{2}{*}{$\{(0,c):c\in \bF\}$}\\

&$k>2$& $\{(0,c):c\in \bF\}$ & \\
\hline
\multirow{3}{*}{$B_p(m,\theta,\varepsilon)$}& $\varepsilon=0$ and $k=2$& $\{(0,0,c):c\in\tup{Im}(f_{1,\theta})\}$ &\multirow{3}{*}{$\{(0,0,c):c\in \bF\}$ }\\
&$\varepsilon=0$ and $k>2$& $\{(0,0,c):c\in \bF\}$ & \\
&$\varepsilon\neq0$ & $\{(0,0,c):c\in \bF\}$ & \\

\hline
\multirow{3}{*}{$C_p(m,\theta,\varepsilon)$}& $\varepsilon=0$ and $k=2$& $\{(0,0,c):c\in\tup{Im}(f_{1,\theta})\}$ &$\{(0,b,c):b,c\in \bF\}$ \\
&$\varepsilon=0$ and $k>2$& $\{(0,0,c):c\in \bF\}$ &$\{(0,b,c):b,c\in \bF\}$  \\
&$\varepsilon\neq0$ & $\{(0,0,c):c\in \bF\}$ &$\{(0,0,c):c\in \bF\}$ \\
\hline
\multirow{3}{*}{$D_p(m,\theta,\varepsilon)$}& $\varepsilon=0$ and $k=2$& $\{(0,0,c):c\in\tup{Im}(f_{1,\theta})\}$ &$\{(0,b,c):b,c\in\bF\}$\\
&$\varepsilon=0$ and $k>2$& $\{(0,0,c):c\in \bF\}$ &\multirow{2}{*}{$\{(0,0,c):c\in \bF\}$ } \\
&$\varepsilon\neq0$ & $\{(0,0,c):c\in \bF\}$ & \\
\hline
\end{tabular}
\end{table}
We shall need the following standard facts from character theory.
 \begin{lemma}\cite[Theorem 17.3]{character theory}\label{lem_nonlinearchar}
Let $H$ be a normal subgroup of a finite group $G$. Then there is a bijection between the set $\tup{Irr}(G/H)$ and the set $\{\chi \in\tup{Irr}(G): H\leqslant\tup{Ker}\chi\}$ by  associating each irreducible character of $G/H$ with its lift to $G$.
\end{lemma}
In particular, if we take $H=G'$, then the linear characters of $G$ are exactly the lifts to $G$ of the irreducible characters of $G/G'$. In addition, the number of linear characters of $G$ is $|\tup{Lin}(G)|=|G/G'|$.
It is known that the number of irreducible characters of a finite group $G$ is equal to the class number of  $G$. So we can deduce that
\begin{equation}\label{k(G)&IRR(G)}
 | \tup{Irr}_{1}(G)|=\tup{k}(G)-|G/G'|,
\end{equation}
where we recall that $\tup{Irr}_{1}(G)$ stands for the set of all non-linear irreducible characters.

In \cite{Lewis(VZ)}, Lewis call a finite group $G$  a \emph{$VZ$-group} if all its non-linear irreducible characters vanish on $G\setminus Z(G)$, that is, $\chi(g)=0$ for all $g\in G\setminus Z(G)$ and all $\chi\in \tup{Irr}_1(G)$.
By \cite[Lemma 2.2]{Lewis(GC)},  a finite non-abelian group $G$ is a $VZ$-group if and only if $[g, G]=G'$ for each $g\notin Z(G)$.
The next theorem  gives the character degrees and the non-linear irreducible characters of a $VZ$-group.
\begin{thm}\cite[Theorem 7.5]{character theory  Huppert}\label{main theorem}
Let $G$ be a $VZ$-group, then
$|G/Z(G)|=m^2$ for some positive integer $m$ and $\tup{cd}(G)=\{1,m\}$. Furthermore, there are $|Z(G)|-|Z(G)/G'|$ distinct non-linear irreducible characters
and all of them can be described by
\[
\chi_m^\lambda(g)=\begin{cases}m\lambda(g),&\tup{ if }g\in Z(G),\\ \quad0,& \tup{    otherwise},\end{cases}
\]
where $\lambda\in \tup{Irr}(Z(G))$, but $\lambda|_{G'}\neq1$.
\end{thm}

\begin{lemma}\label{lem_Vz}
 Write $o(\theta)=k$. In Table \ref{table_VZ},  we list the cases where the Suzuki $p$-group is a $VZ$-group.
  \begin{table}[H]
\caption{Cases where the Suzuki $p$-group is a $VZ$-group}\label{table_VZ}
\begin{tabular}{c|c}
\hline
 \textup{type} &\textup{condition}  \\
 \hline
$A_p(m,\theta)$&$k=2$\\

$B_p(m,\theta,\varepsilon)$&$\varepsilon\neq0$ or $\varepsilon=0,k=2$\\

$C_p(m,\theta,\varepsilon)$&$\varepsilon\neq0$ or $\varepsilon=0,k=2$\\

$D_p(m,\theta,\varepsilon)$&$\varepsilon\neq0$ or $\varepsilon=0,k=2$\\
\hline
\end{tabular}
\end{table}
\end{lemma}
\begin{proof}
  We first consider the case $G=A_p(m,\theta)$. Recall that when $k=2$, it follows from Table \ref{table_Z(G)&G'} that
  \[
  G'=\{(0,b):b\in \tup{Im}(f_{1,\theta})\},\,Z(G)=\{(0,b):b\in\bF\}.\]
  For each $(a,b)\in G\setminus Z(G)$, we have $[(a,b),G]=\{(0,x):x\in\tup{Im}(f_{a,\theta})\}=G'$ by Table \ref{table_commutator} and Lemma $\ref{Im(f_a)} \,(2)$. So $G$ is a $VZ$-group.

  Next we  consider the case $G=B_p(m,\theta,\varepsilon)$. Recall that when $\varepsilon\neq 0$, we have $Z(G)=G'=\{(0,0,c):c\in\bF\}$  by Table \ref{table_Z(G)&G'}.
  For each $(a,b,c)\in G\setminus Z(G)$, we have
   \[[(a,b,c),G]=\{(0,0,f_{a,\theta}(d)+f_{b,\theta}(e)+\varepsilon(ae^\theta-db^\theta): d,e\in\bF\}\] by Table \ref{table_commutator}.
  When $a\neq 0$ or $b\neq 0$, it follows that \[\{ae^\theta-db^\theta: b,d\in\bF\}=\bF. \]
   So we have $[(a,b,c),G]=G'$ and $G$ is a $VZ$-group.
  When $\varepsilon=0$, this is the case similar to $G=A_p(m,\theta)$ and the claim follows similarly.
  As for the remaining two cases, they are similar to $G=B_p(m,\theta,\varepsilon)$ and we omit  the details.
\end{proof}
In the remaining of this section, we only consider the cases where the Suzuki $p$-groups are not  $VZ$-groups.

  \begin{lemma}\label{lem fa&fb^2}
    Suppose $\gamma$ is  a primitive element  of $\bF^*$. Let $a,b\in\bF^*$ and take $\theta\in\tup{Aut}(\bF)$ of order $k>2$. Then the following holds:
    \begin{enumerate}
     \item When $k$ is odd, for each $a\in \bF^*$, there exists  $b\in \bF^*$ such that $\tup{Im}(f_{b,\theta^2})\subseteq\tup{Im}(f_{a,\theta})$, and here are $p^n-1$ such $b$'s.
       \item   When  $k/2$ is odd, for each $a\in \bF^*$, there exists  $b\in \bF^*$ such that $\tup{Im}(f_{b,\theta^2})\subseteq\tup{Im}(f_{a,\theta})$, and here are  $p^{2n}-1$ such $b$'s.
      \item When $k/2$ is even, there exists some $b\in \bF^*$ such that $\tup{Im}(f_{b,\theta^2})\subseteq\tup{Im}(f_{a,\theta})$ if and only if $a\in \langle \gamma^{\frac{p^{2n}+1}{\gcd(2,p-1)}}\rangle$, and here are $p^{4n}-1$ such $b$'s if this is the case.
    \end{enumerate}
  \end{lemma}
  \begin{proof}

  Suppose that $\tup{Im}(f_{b,\theta^2})\subseteq\tup{Im}(f_{a,\theta})$ for some $b\in\bF^*$. This is the case if and only if
  \[
  \tup{Tr}_{\bF/\bF_\theta}((aa^\theta)^{-1}(bx^{\theta^2}-xb^{\theta^2}))=0,\, \forall x\in \bF,
  \]  where recall that $\tup{Im}(f_{a,\theta})=\{x\in\bF:\tup{Tr}_{\bF/\bF_{\theta}}((aa^{\theta})^{-1}x)=0\}$ from Lemma \ref{Im(f_a)} $(1)$.
That is, \[\tup{Tr}_{\bF/\bF_\theta}(((aa^\theta)^{-1}b)^{\theta^{-2}}-(aa^\theta)^{-1}b^{\theta^2})x)=0, \forall x\in \bF.\] It follows that
\[((aa^\theta)^{-1}b)^{\theta^{-2}}=(aa^\theta)^{-1}b^{\theta^2},\, i.e.,\, b^{\theta^4}b^{-1}=a^{\theta^3}a^{\theta^2}(aa^\theta)^{-1}.\]
Write $x^\theta=x^{p^l}$ and set $d:=\gcd(4l,m)$. Observe that $\{b^{p^{4l}-1}: b\in\bF^*\}$ is  a multiplicative subgroup of $\bF^*$ of order
\[\frac{p^m-1}{\gcd(p^{4l-1},p^m-1)}=\frac{p^m-1}{p^d-1},\]
so $\{b^{p^{4l}-1}: b\in\bF^*\}=\langle\gamma^{p^d-1}\rangle$.
We conclude that there exists some $b\in\bF^*$ such that $\tup{Im}(f_{b,\theta^2})\subseteq\tup{Im}(f_{a,\theta})$ if and only if
\begin{equation*}
  a^{p^{3l}+p^{2l}-p^l-1}=a^{(p^l+1)(p^{2l}-1)}\in  \langle\gamma^{p^d-1}\rangle.
\end{equation*}
It implies that
\begin{equation}\label{eq_ImfbIm(fa)}
  p^d-1 \mid s(p^l+1)(p^{2l}-1),
\end{equation}
where $a=\gamma^s$.
We now consider the three cases one by one.
\begin{enumerate}
  \item In the case $d=n$, $k$ is odd. We have \eqref{eq_ImfbIm(fa)} holds for any $a\in\bF^*$ since $p^n-1 \mid p^{2l}-1$. Besides, as $\tup{Im}(f_{a,\theta})$ can not contain two distinct $\tup{Im}(f_{b,\theta^2})$'s,  the number of $b$'s such that $\tup{Im}(f_{b,\theta^2})\subseteq\tup{Im}(f_{a,\theta})$  is equal to $|\bF_{\theta^4}^*|$ by Lemma \ref{Im(f_a)} $(2)$.
      Since $o(\theta^4)=k$, we have $\bF_{\theta^4}^*=\bF_{p^n}^*$. So there are $p^n-1$ such $b$'s.
  \item In the case $d=2n$, $k/2$ is odd. Similarly, \eqref{eq_ImfbIm(fa)} holds for any $a\in\bF^*$ as $p^{2n}-1 \mid p^{2l}-1$.
        Besides there are $p^{2n}-1$ such $b$'s since $\bF_{\theta^4}^*=\bF_{p^{2n}}^*$ in this case.
  \item In the case $d=4n$, $k/2$ is even and \eqref{eq_ImfbIm(fa)} is equivalent to $ p^{2n}+1 \mid s(p^l+1)$.
      Write $l=nr$. Since $\gcd(k,r)=1$, we deduce that $r$ is odd.
       Then we have \[\gcd(p^{2n}+1,p^{l}+1)=\gcd(2,p-1).\]
So \eqref{eq_ImfbIm(fa)} holds if and only if $\frac{p^{2n}+1}{\gcd(2,p-1)} \mid s$, that is, $a=\gamma^s\in \< \gamma^{\frac{p^{2n}+1}{\gcd(2,p-1)}}\>$.
Since $o(\theta^4)=k/4$ in this case, we have $\bF_{\theta^4}^*=\bF_{p^{4n}}^*$. So there are $p^{4n}-1$ such $b$'s.
\end{enumerate}
  \end{proof}

\begin{lemma} \label{k(G)of four types}
The class number of the  Suzuki $p$-groups that are not $VZ$-groups  is given in Table \ref{table_k(G)}.
\begin{table}[H]
\begin{center}
\caption{The class number}\label{table_k(G)}
\begin{tabular}{c|c|c}
\hline
 $G$ & Condition&$\tup{k}(G)$  \\
 \hline
$A_p(m,\theta)$& &$p^m+ p^{n}(p^m-1)$\\
\hline
\multirow{2}{*}{$B_p(m,\theta,0)$}& $k>1$ odd&$p^{2m}+ p^{2n}(p^m-1)$\\
\cline{2-3}
&$k>2$ even&$p^{2m}+ p^{n}(p^m-1)(p^{2n}-p^n+1)$\\
\hline
$C_p(m,\theta,0)$&$k>2$&$p^{2m}+p^{m+n}(p^m-1)$\\
\hline
\multirow{4}{*}{$D_p(m,\theta,0)$}&$k>1$ odd&$p^{2m}+ p^{2n}(p^m-1)$\\
\cline{2-3}
& $k$ even and $k/2>1$ odd&$p^{2m}+ p^{3n}(p^m-1)$\\
\cline{2-3}
&\multirow{2}{*}{$k/2$ even} &$2^{2m}+ 2^{3n}(2^m-1)$ or \\
&&$p^{2m}+(p^m-1)(2p^{3n}-p^{2n}-p^n+1)$ \\
\hline
\end{tabular}
\end{center}
\end{table}
\end{lemma}
\begin{proof}

We only give the details in the case  $G=D_p(m,\theta,0)$, since the cases where $G=A_p(m,\theta), B_p(m,\theta,0)$ and $C_p(m,\theta,0)$ are similar and easier.
It is deduced from Table \ref{table_commutator} that
\[
(a,b,c)^{G}=\{(a,b,c+x):x\in \tup{Im}(f_{a,\theta})+ \tup{Im}(f_{b,\theta^2})\}.
\]
Recall that when $k>2$,  $\tup{Im}(f_{a,\theta})$ is an $\bF_{\theta}$-hyperplane if $a\neq 0$ and $\tup{Im}(f_{b,\theta^2})$ is an $\bF_{\theta^2}$-hyperplane if $b\neq 0$. So we have that $\tup{Im}(f_{a,\theta})+\tup{Im}(f_{b,\theta^2})$ equals $\tup{Im}(f_{a,\theta})$ or $\bF$ according as $\tup{Im}(f_{b,\theta^2})\subseteq\tup{Im}(f_{a,\theta})$ or not.
In the former case, we  get $p^n$ distinct conjugacy classes whose elements have $(a,b)$ as their first two coordinates.
In the later case, we have exactly one conjugate class whose elements have $(a,b)$ as their first two coordinates.
By Lemma \ref{lem fa&fb^2}, we list the number of conjugate classes in each case in the following table. We obtain $\tup{k}(G)$ by adding up the numbers in all cases. This completes the proof for the case $G=D_p(m,\theta,0)$.
\begin{table}[H]
\resizebox{\textwidth}{!}{
\begin{tabular}{c|ccccc}

\hline
\multirow{2}{*}{$k$}& $a=0$& $a\neq 0$&$a=0$&$a\neq 0,b\neq0,$ &$a\neq 0,b\neq0,$ \\

&$b=0$&$ b=0$ &$b\neq 0$ &$\tup{Im}(f_{b,\theta^2})\subseteq\tup{Im}(f_{a,\theta})$&$\tup{Im}(f_{b,\theta^2})\nsubseteq\tup{Im}(f_{a,\theta})$\\
 \hline
$k>1$ odd& $p^m$&$(p^m-1)p^n$ &$(p^m-1)p^n$&$(p^m-1)(p^n-1)p^n$& $(p^m-1)(p^m-p^n)$\\
\hline
$k/2>1$ odd&$p^m$&$(p^m-1)p^n$&$(p^m-1)p^{2n}$&$(p^m-1)(p^{2n}-1)p^n$&$(p^m-1)(p^m-p^{2n})$\\

\hline
\multirow{2}{*}{$k/2>1$ even}& \multirow{2}{*}{$p^m$} &\multirow{2}{*}{$(p^m-1)p^n$} &\multirow{2}{*}{$(p^m-1)p^{2n}$}&$(2^m-1)(2^{2n}-1)2^n$&$(2^m-1)(2^m-2^{2n})$\\
&&&&$2(p^m-1)(p^{2n}-1)p^n$&$(p^m-1)(p^m-2p^{2n}+1)$\\
\hline
\end{tabular}}
\end{table}

  \end{proof}

\section{the main results }\label{main results}
Set  $\bF=\bF_{p^m}$, $\xi_p=\exp(\frac{2\pi\sqrt{-1}}{p})$ and recall that $n=m/k$. Define
  \begin{equation}\label{psi_v}
    \psi_v(x)=\xi_p^{\textup{Tr}_m(vx)}, \forall x,v\in\bF,
  \end{equation}
and
    \begin{equation}\label{phi_w}
    \phi_w(y)=\xi_p^{\textup{Tr}_n(wy)}, \forall y,w\in\bF.
  \end{equation}
   It is easy to see that \[\tup{Irr}(\bF,+)=\{\psi_v: v\in \bF\}\tup{ and }\tup{Irr}(\bF_{p^n},+)=\{\phi_w: w\in \bF_{p^n}\}.\]
    For more information about the additive characters of a finite fields, please refer to \cite[Chapter 5]{finite fields}.

Let $T$ be a set of coset representatives  for $\bF_p^*$ in $\bF^*$.
 Set
 \begin{equation}\label{eq_J1}
   J_1=\{v\in\bF^*:\tup{Im}(f_{a,\theta})\nsubseteq\tup{Ker}(\psi_v)\tup{ for any }a\in\bF^*\}\tup{ and } J_2=\bF^*\setminus J_1.
 \end{equation}
  Similarly, set
  \begin{equation}\label{eq_S1}
S_1=\{v\in\bF^*:\tup{Im}(f_{b,\theta^2})\nsubseteq \tup{Ker}(\psi_v) \tup{ for any }b\in\bF^*\} \tup{ and } S_2=\bF^*\setminus S_1.
 \end{equation}
 In the case $G= C_p(m,\theta,\varepsilon)$, we define some sets that will be used in the main theorem.
   \begin{enumerate}
     \item When $k>1$ is odd, we define the set $I$ as follows:
     \begin{enumerate}
       \item If $p>2$, let $\alpha^{(v,w)}$ be as in \eqref{vartheta^{u,v}}.  Regard $\tup{PG}(\bF^2)$ as a projective space over $\bF_p$, and denote by $[v,w]$ the projective point corresponding to the non-zero vector $(v,w)$. Define
  \begin{equation}\label{eq_I,odd}
    I=\{\alpha^{(v,w)}\in \tup{Lin}(Z(G)):[v,w]\in \tup{PG}(\bF^2) \tup{ with } v\neq 0 \}.
  \end{equation}
       \item If $p=2$, let $\alpha^{(v,w,1)}$ be as in \eqref{vartheta^{v,w,e}}. Define
    \begin{equation}\label{eq_I,2}
    I=\{\alpha^{(v,w,1)}\in \tup{Lin}(Z(G)):v\in\bF^*,w\in\bF/\bF_2\}.
  \end{equation}
     \end{enumerate}

     \item When $k>2$ is even, we define the   sets $I_1,I_2$ as follows:
      \begin{enumerate}
        \item If $p>2$, define \[I_i=\{\alpha^{(v,w)}\in \tup{Lin}(Z(G)): [v,w]\in \tup{PG}(\bF^2) \tup{ with } v\in J_i\}, \,i=1,2.\]
        \item  If $p=2$, define \[I_i=\{\alpha^{(v,w,1)}\in \tup{Lin}(Z(G)):v\in J_i,w\in\bF/\bF_2\},\, i=1,2.\]
      \end{enumerate}
   \end{enumerate}
 We now state our main theorem below, whose proof will be given in Section \ref{pf_main result}.

 \begin{thm}\label{thm 1}
   Let $G$ be one of the four families of  \textbf{ Suzuki $p$-groups}. Write  $o(\theta)=k$.
   Then  $\tup{Irr}(G)$ is given in Table \ref{tableA}-\ref{tableD} respectively.
 \end{thm}

 \begin{table}[H]
\caption{Irreducible characters of $G:=A_p(m,\theta)$}\label{tableA}
\resizebox{\textwidth}{!}{
\begin{tabular}{ccccccc}
\hline\hline
 Case &Degree & Family&  Number &Character&Restriction  \\
 \hline\hline
\multirow{2}{*}{k=2}&$1$& $\tup{Lin}(G)$& $p^{m+n}$&   $\chi_1^{\beta}(a,b)=\beta\overline{(a,b)},\overline{(a,b)}\in G/G' $& $\beta\in \tup{Lin}(G/G')$ \\

&$p^n$ &$\tup{Irr}_{(p^n)}(G)$& $p^m-p^{n}$&$\chi_{p^n}^v(a,b)=
\begin{cases}
  p^n\psi_v(b) &\tup{ if } a=0\\
  0  &\tup{ otherwise}
\end{cases}$ &$v\in \bF\setminus\bF_{p^n}$
\\
\hline
\multirow{2}{*}{$k >1$ odd}&$1$& $\tup{Lin}(G)$& $p^{m}$& $\chi_1^v(a,b)=\psi_v(b)$ &$v\in\bF$\\

&$p^{\frac{m-n}{2}}$ &$\tup{Irr}_{(p^{\frac{m-n}{2}})}(G)$& $p^n(p^m-1)$& $\chi_{p^{\frac{m-n}{2}}}^{(v,\lambda)}=\eqref{eq_nonlin(v,lam2)},r=p^{\frac{m-n}{2}}$&$ v\in T,\lambda\in \tup{Lin}(Z(G_v)),\lambda|_{G'_v}\neq 1$\\
\hline
\multirow{3}{*}{$k >2$ even}&$1$& $\tup{Lin}(G)$ &$p^{m}$& $\chi_1^{v}(a,b)=\psi_v(b)$ &$v\in\bF$\\

&$p^{\frac{m}{2}}$& $\tup{Irr}_{(p^{\frac{m}{2}})}(G)$& $\frac{p^n(p^m-1)}{p^n+1}$& $\chi_{p^{m/2}}^{(v,\lambda)}=\eqref{eq_nonlin(v,lam2)},r=p^{m/2}$ &$v\in J_1\cap T,1\neq\lambda\in \tup{Lin}(Z(G_v))$ \\

&$p^{\frac{m-2n}{2}}$ &$\tup{Irr}_{(p^{\frac{m-2n}{2}})}(G)$& $\frac{p^{2n}(p^m-1)}{p^n+1}$ &  $\chi_{p^{(m-2n)/2}}^{(v,\lambda)}=\eqref{eq_nonlin(v,lam2)},r=p^{(m-2n)/2}$ &$v\in J_2\cap T,\lambda\in\tup{Lin}(Z(G_v)),\lambda|_{G'_v}\neq 1$\\
\hline

\end{tabular}}
\end{table}

\begin{table}[H]
\caption{Irreducible characters of $G:=B_p(m,\theta,\varepsilon)$}\label{tableB}
\resizebox{\textwidth}{!}{
\begin{tabular}{ccccccc}
\hline \hline
Case & Degree & Family &  Number &Character &Restriction\\
 \hline \hline
\multirow{2}{*}{ $\varepsilon\neq 0$}&$1$& $\tup{Lin}(G)$& $p^{2m}$& $  \chi_1^{(v,w)}(a,b,c)=\eqref{eq_linuni} $& $v,w\in\bF$\\

&$p^m$ &$\tup{Irr}_{(p^m)}(G)$ &$p^m-1$& $  \chi_{p^m}^v(a,b,c)=
 \begin{cases}
   p^m\psi_v(a,b,c)&\tup{ if } {a=b=0}\\
   \,0& \tup{ otherwise}
 \end{cases}$ &$v\in \bF^*$\\
\hline
\multirow{2}{*}{ $\varepsilon=0, k=2$}&$1$&$\tup{Lin}(G)$& $p^{2m+n}$&$\chi_1^{\beta}(a,b,c)=\beta\overline{(a,b,c)},\overline{(a,b,c)}\in G/G' $ &$\beta\in Lin(G/G')$\\

&$p^{m}$ &$\tup{Irr}_{(p^{m})}(G)$& $p^m-p^n$&  $\chi_{p^m}^v(a,b,c)=
 \begin{cases}
   p^m\psi_v(a,b,c)&\tup{ if } {a=b=0}\\
   \,0& \tup{ otherwise}
 \end{cases}$ &$v\in\bF\setminus\bF_{p^n}$\\
\hline
\multirow{2}{*}{ $\varepsilon= 0, k>1$ odd}&$1$& $\tup{Lin}(G)$& $p^{2m}$& $\chi_1^{(v,w)}=\eqref{eq_linuni}$ &$v,w\in\bF$\\

&$p^{m-n}$ &$\tup{Irr}_{(p^{m-n})}(G)$&$p^{2n}(p^m-1)$ &  $\chi_{p^{m-n}}^{(v,\lambda)}=\eqref{eq_nonlin(v,lam)},r=p^{m-n}$ &$v\in T,\lambda\in \tup{Lin}(Z(G_v)),\lambda|_{G'_v}\neq 1$ \\

\hline
\multirow{2}{*}{ $\varepsilon= 0, k>1$ even}&$1$& $\tup{Lin}(G)$& $p^{2m}$& $\chi_1^{(v,w)}=\eqref{eq_linuni}$ &$v,w\in\bF$\\

&$p^{m}$& $\tup{Irr}_{(p^{m})}(G)$&$\frac{p^n(p^m-1)}{p^n+1}$& $\chi_{p^m}^{(v,\lambda)}=\eqref{eq_nonlin(v,lam)},r=p^m$ &$v\in J_1\cap T,1\neq\lambda\in \tup{Lin}(Z(G_v))$ \\

&$p^{m-2n}$ &$\tup{Irr}_{(p^{m-2n})}(G)$& $\frac{p^{4n}(p^m-1)}{p^n+1}$ &  $\chi_{p^{m-2n}}^{(v,\lambda)}=\eqref{eq_nonlin(v,lam)},r=p^{m-2n}$ &$v\in J_2\cap T,\lambda\in\tup{Lin}(Z(G_v)),\lambda|_{G'_v}\neq 1$\\
\hline

\end{tabular}}
\end {table}

\begin{table}[H]

\caption{Irreducible characters of $G:=C_p(m,\theta,\varepsilon)$}\label{tableC}

\resizebox{\textwidth}{!}{
\begin{tabular}{ccccccc}
\hline\hline
Case & Degree & Family&  Number &Character &Restriction\\
\hline\hline
\multirow{2}{*}{$\varepsilon=0,k=2$}& 1 &$\tup{Lin}(G)$&$p^{2m+n}$& $\chi_1^{\beta}=\beta\overline{(a,b,c)},\overline{(a,b,c)}\in G/G'$ &$\beta\in \tup{Lin}(G/G')$ \\
&$p^{\frac{m}{2}}$ &$\tup{Irr}_{(p^{\frac{m}{2}})}(G)$ & $p^{m}(p^m-p^n)$&  $\chi_{p^{\frac{m}{2}}}^{\beta}(a,b,c)=
\begin{cases}
  p^{\frac{m}{2}}\beta\overline{(a,b,c)} &\tup{ if } \overline{(a,b,c)}\in Z(G)\\
  0 &\tup{ otherwise }
\end{cases}
$ &$\beta\in \tup{Lin}(Z(G)),\beta|_{G'}\neq 1$\\
\hline
\multirow{2}{*}{$\varepsilon=0,k>1$ odd}&$1$& $\tup{Lin}(G)$& $p^{2m}$& $\chi_1^{(v,w)}=\eqref{eq_linuni}$ &$v,w\in\bF$\\

&$p^{\frac{m-n}{2}}$&$\tup{Irr}_{(p^{\frac{m-n}{2}})}(G)$&$p^{m+n}(p^m-1)$&  $\chi_{p^{\frac{m-n}{2}}}^{(\alpha,\lambda)}=\eqref{eq_nonlinuni2},r=p^{\frac{m-n}{2}}$ &$\alpha\in I, \lambda\in \tup{Lin}(Z(G_\alpha)) ,\lambda|_{G'_\alpha}\neq 1$\\
\hline
\multirow{3}{*}{$\varepsilon=0,k>2$ even}&$1$& $\tup{Lin}(G)$&$p^{2m}$& $\chi_1^{(v,w)}=\eqref{eq_linuni}$ &$v,w\in\bF$\\

&$p^{\frac{m}{2}}$ &$\tup{Irr}_{(p^{\frac{m}{2}})}(G)$& $\frac{p^{m+n}(p^m-1)}{p^n+1}$&  $\chi_{p^{\frac{m}{2}}}^{(\alpha,\lambda)}=\eqref{eq_nonlinuni2},r=p^{\frac{m}{2}}$ &$\alpha\in I_1, 1\neq\lambda\in \tup{Lin}(Z(G_\alpha))$\\
&$p^{\frac{m-2n}{2}}$ &$\tup{Irr}_{(p^{\frac{m-2n}{2}})}(G)$& $\frac{p^{m+2n}(p^m-1)}{p^n+1}$&  $\chi_{p^{\frac{m-2n}{2}}}^{(\alpha,\lambda)}=\eqref{eq_nonlinuni2},r=p^{\frac{m-2n}{2}}$ &$\alpha\in I_2,\lambda\in \tup{Lin}(Z(G_\alpha)),\lambda|_{G'_\alpha}\neq 1$\\
\hline
\end{tabular}}
\end{table}
\begin{table}[H]
\caption{Irreducible characters of $G:=D_p(m,\theta,0)$}\label{tableD}
\resizebox{\textwidth}{!}{
\begin{tabular}{cccccccc}
\hline\hline
Case& Degree & Family& \multicolumn {2}{c}{ Number} &Character &Restriction\\
 \hline\hline
\multirow{2}{*}{$k>1$ odd}&$1$& $\tup{Lin}(G)$& \multicolumn {2}{c}{$p^{2m}$}& $\chi_1^{(v,w)}=\eqref{eq_linuni}$ &$v,w\in\bF $\\

&$p^{m-n}$ &$\tup{Irr}_{(p^{m-n})}(G)$&\multicolumn {2}{c}{$p^{2n}(p^m-1)$}& $\chi_{p^{m-n}}^{(v,\lambda)}=\eqref{eq_nonlin(v,lam)},r=p^{m-n}$ &$v\in T,\lambda\in \tup{Lin}(Z(G_v)),\lambda|_{G'_v}\neq 1$ \\
\hline
\multirow{3}{*}{$k/2>1$ odd}&$1$& $\tup{Lin}(G)$&\multicolumn {2}{c}{ $p^{2m}$}& $\chi_1^{(v,w)}=\eqref{eq_linuni}$ &$v,w\in\bF$\\

&$p^{m-n}$& $\tup{Irr}_{(p^{m-n})}(G)$&\multicolumn {2}{c}{$\frac{p^{3n}(p^m-1)}{p^n+1}$}& $\chi_{p^{m-n}}^{(v,\lambda)}=\eqref{eq_nonlin(v,lam)},r=p^{m-n}$ &$v\in J_1\cap T,\lambda\in \tup{Lin}(Z(G_v)),\lambda|_{G'_v}\neq 1$\\

&$p^{m-2n}$ &$\tup{Irr}_{(p^{m-2n})}(G)$& \multicolumn {2}{c}{$\frac{p^{4n}(p^m-1)}{p^n+1}$}&  $\chi_{p^{m-2n}}^{(v,\lambda)}=\eqref{eq_nonlin(v,lam)},r=p^{m-2n}$ &$v\in J_2\cap T,\lambda\in\tup{Lin}(Z(G_v)),\lambda|_{G'_v}\neq 1$\\
\hline
\multirow{5}{*}{$k/2>1$ even}&$1$& $\tup{Lin}(G)$&\multicolumn {2}{c}{$p^{2m}$}& $\chi_1^{(v,w)}=\eqref{eq_linuni}$ &$v,w\in\bF$\\

&$p^{m}$& $\tup{Irr}_{(p^{m})}(G)$&$\frac{2^{3n}(2^m-1)}{(2^n+1)(2^{2n}+1)}$&$\frac{(p^{3n}+1)(p^m-1)}{(p^n+1)(p^{2n}+1)}$  &$\chi_{p^{m}}^{(v,\lambda)}=\eqref{eq_nonlin(v,lam)},r=p^{m}$ &$v\in S_1\cap J_1\cap T,1\neq\lambda\in \tup{Lin}(Z(G_v))$\\

&$p^{m-n}$& $\tup{Irr}_{(p^{m-n})}(G)$& $\frac{2^{4n}(2^m-1)}{(2^n+1)(2^{2n}+1)}$& $\frac{(p^{2n}-1)(p^m-1)p^{2n}}{(p^n+1)(p^{2n}+1)}$ & $\chi_{p^{m-n}}^{(v,\lambda)}=\eqref{eq_nonlin(v,lam)},r=p^{m-n}$ &$v\in S_1\cap J_2\cap T,\lambda\in \tup{Lin}(Z(G_v)),\lambda|_{G'_v}\neq 1$\\

&$p^{m-2n}$& $\tup{Irr}_{(p^{m-2n})}(G)$&$\frac{2^{5n}(2^m-1)}{(2^n+1)(2^{2n}+1)}$ & $\frac{(p^{n}-1)(p^m-1)p^{4n}}{(p^n+1)(p^{2n}+1)}$& $\chi_{p^{m-2n}}^{(v,\lambda)}=\eqref{eq_nonlin(v,lam)},r=p^{m-2n}$ &$v\in S_2\cap J_1\cap T,\lambda\in \tup{Lin}(Z(G_v)),\lambda|_{G'_v}\neq 1$ \\

&$p^{m-3n}$ &$\tup{Irr}_{(p^{m-3n})}(G)$&$\frac{2^{6n}(2^m-1)}{(2^n+1)(2^{2n}+1)}$ & $\frac{(2p^{6n})(p^m-1)}{(p^n+1)(p^{2n}+1)}$&  $\chi_{p^{m-3n}}^{(v,\lambda)}=\eqref{eq_nonlin(v,lam)},r=p^{m-3n}$ &$v\in S_2\cap J_2\cap T,\lambda\in\tup{Lin}(Z(G_v)),\lambda|_{G'_v}\neq 1$\\
\hline
\end{tabular}}
\end{table}

  \begin{remark}
 \begin{enumerate}
 \item In the cases $\varepsilon\neq 0$ and $G=C_p(m,\theta,\varepsilon)$ or $D_p(m,\theta,\varepsilon)$, $\tup{Irr}(G)$ is the same as that in the case $G=B_p(m,\theta,\varepsilon)$ with $\varepsilon\neq 0$. Besides, In the case $G=D_p(m,\theta,0)$ with $k=2$, $\tup{Irr}(G)$ is the same as that in the case $G=C_p(m,\theta,0)$ with $k=2$. So we omit those cases in the corresponding tables.
  \item We can read $\tup{cd}(G)$ from the second column of the following tables.
  \item We give the equations indicated in Table \ref{tableA}-\ref{tableD} as follows:
  \begin{equation}\label{eq_linuni}
  \chi_1^{(v,w)}(a,b,c)=\psi_v(a)\psi_w(b).
\end{equation}
\begin{equation}\label{eq_nonlin(v,lam2)}
   \chi_{r}^{(v,\lambda)}(a,b)=
           \begin{cases}
            r\lambda\overline{(a,b)},& \tup{ if } \overline{(a,b)}\in Z(G_v),\\
             \,\,0, &\tup{ otherwise}.
        \end{cases}
 \end{equation}
   \begin{equation}\label{eq_nonlin(v,lam)}
   \chi_{r}^{(v,\lambda)}(a,b,c)=
           \begin{cases}
            r\lambda\overline{(a,b,c)},& \tup{ if } \overline{(a,b,c)}\in Z(G_v),\\
             \,\,0, &\tup{ otherwise}.
        \end{cases}
 \end{equation}
  \begin{equation}\label{eq_nonlinuni2}
   \chi_{r}^{(\alpha,\lambda)}(a,b,c)=
           \begin{cases}
            r\alpha\overline{(a,b,c)},& \tup{ if } \overline{(a,b,c)}\in Z(G_\alpha),\\
            \,\, 0, &\tup{ otherwise}.
        \end{cases}
 \end{equation}
 Here the notation $G_v$ and $G_\alpha$ are defined in Lemma \ref{G_v}.
 \end{enumerate}
 \end{remark}

\section{proof of theorem \ref{thm 1} }\label{pf_main result}
 We now prove our main results given in Theorem \ref{thm 1}.
Since the irreducible characters of a $VZ$-group is determined by Theorem \ref{main theorem}, we only consider the Suzuki $p$-group which is not a $VZ$-group in this section.
 Recall that $\bF=\bF_{p^m}$, $k:=o(\theta)$ and $n=m/k$.
 \begin{lemma}\label{lem_fainv}
 Let $v\in \bF^*$ and $a\in\bF$.  Then   $\tup{Im}(f_{a,\theta})\subseteq \tup{Ker}(\psi_v)$ if and only if $va^{\theta+1}\in \bF_{\theta}$.
\end{lemma}

\begin{proof}
 Notice that $\tup{Im}(f_{a,\theta})\subseteq \tup{Ker}(\psi_v)$ if and only if $\tup{Tr}_m(v(ax^\theta-xa^\theta))=0$ for all $x\in\bF$. The later holds if and only if
 $va=(v a^{\theta})^\theta$, i.e.,  $va^{\theta+1}=(va^{\theta+1})^\theta$. So the claim follows.
\end{proof}
\begin{lemma}\label{lem(p^l+1,p^m-1)}
  Set $n:=\gcd(l,m)$ and $k:=m/n$. Then
  \[
  \gcd(p^l+1,p^m-1)=
  \begin{cases}
    \,1, & \tup{ if } k \tup{ is odd and } p=2,\\
   \, 2, & \tup{ if } k \tup{ is odd and } p \tup{ is odd},\\
    p^n+1, & \tup{ if } k \tup{ is even}.
  \end{cases}
  \]
\end{lemma}
\begin{lemma}\label{lem_JI&J_2}
 Let $J_i, S_i, \,i=1,2$ be as in \eqref{eq_J1} and \eqref{eq_S1} respectively. Suppose that $\gamma$ is a primitive element of $\bF^*$. Then the following holds:
  \begin{enumerate}
    \item If $k$ is odd,  then $J_1=\varnothing$ and $J_2=\bF^*$.
    \item If $k$ is even, then
      $J_1=\bF^*\setminus\langle \gamma^{p^n+1}\rangle$ and $J_2=\langle \gamma^{p^n+1}\rangle$.
    \item If $k\equiv 2 \mod 4$, then $S_1=\varnothing$ and $S_2=\bF^*$.
    \item If $k\equiv 0 \mod 4$, then
      $S_1=\bF^*\setminus\langle \gamma^{p^{2n}+1}\rangle$ and $J_2=\langle \gamma^{p^{2n}+1}\rangle$.
  \end{enumerate}

\end{lemma}

\begin{proof}
Note that $(3),(4)$ can be deduced from $(1),(2)$ according to the definitions of $S_i$ and $J_i$. So we only prove the claims on  $J_i$.
If $v\in J_2$, then there exists $a\in\bF^*$ such that $\tup{Im}(f_{a,\theta})\subseteq\tup{Ker}(\psi_v)$. This is equivalent to $v\in(aa^\theta)^{-1}\bF_\theta^*$ by Lemma \ref{lem_fainv}.
  We conclude that
  \[
  J_2=\{xy^{1+\theta}: x\in\bF_\theta^*,y\in\bF^*\}.
  \]
It is a multiplicative subgroup of $\bF^*$, and it has order
\[
c:=\tup{lcm}(p^n-1,\frac{p^m-1}{\gcd(p^m-1,p^l+1)}), \tup{ where } x^\theta=x^{p^l}.\]
Then from Lemma \ref{lem(p^l+1,p^m-1)} we have $c=p^m-1$ and $J_2=\bF^*$ if $k$ is odd.
  Otherwise we have $c=\frac{p^m-1}{p^n+1}$ and $J_2=\langle \gamma^{p^n+1}\rangle$. Thus the proof is completed.
\end{proof}
Observe that the fourth column of Table \ref{tableA}- \ref{tableD} follows directly from the last column of those tables and 
the cardinality of $I$, $I_i$, $J_i$ and $S_i, i=1,2$ is deduced easily from their definitions and Lemma \ref{lem_JI&J_2}. We now give
 the cardinality of $S_2\cap J_2$, $S_1\cap J_2$, $S_2\cap J_1$ and $S_2\cap J_2$ with $k\equiv0 \mod 4$. 

For $v\in S_2\cap J_2$, there exist some $a,b\in\bF^*$ such that $\tup{Im}(f_{a,\theta})\subseteq \tup{Ker}(\psi_v)$ and $\tup{Im}(f_{b,\theta^2})\subseteq \tup{Ker}(\psi_v)$, that is,
\[
\tup{Im}(f_{a,\theta})+\tup{Im}(f_{b,\theta^2})\subseteq \tup{Ker}(\psi_v),\, i.e., \,\tup{Im}(f_{b,\theta^2})\subseteq\tup{Im}(f_{a,\theta})\subseteq \tup{Ker}(\psi_v).
\]
Recall that there are $p^{2n}-1$ distinct $a$'s in $\bF^*$ such that $\tup{Im}(f_{a,\theta})$ are the same according to Lemma \ref{Im(f_a)} (2). Combining it with Lemma \ref{lem fa&fb^2} (3),   we conclude that there exists $\frac{\gcd(2,p-1)(p^m-1)}{p^{4n}-1}$ distinct $\tup{Im}(f_{a,\theta})$'s which contain some $\tup{Im}(f_{b,\theta^2})$. Note that there are $p^n-1$ distinct $v$'s in $\bF^*$ such that $\tup{Ker}(\psi_v)$ contain the same $\tup{Im}(f_{a,\theta})$.
 Thus the cardinality of $S_2\cap J_2$ is equal to
 \[
 \frac{\gcd(2,p-1)(p^m-1)}{p^{4n}-1}(p^n-1)=\frac{\gcd(2,p-1)(p^m-1)}{(p^{2n}+1)(p^n+1)}.
 \]
As for the remaining  three sets, their cardinalities can follow from $|S_2\cap J_2|$.
We list the results on those cardinalities in the following table.
 \begin{table}[H]
\caption{Cardinalities of certain sets with $k/2>1$ even}\label{table(card)}
\begin{tabular}{ccccc}
\hline
 condition & $S_2\cap J_2$ & $S_1\cap J_2$&$S_2\cap J_1$&  $S_1\cap J_1$\\
 \hline
$p=2$& $\frac{p^m-1}{(p^{2n}+1)(p^n+1)}$&  $\frac{(p^m-1)p^{2n}}{(p^{2n}+1)(p^n+1)}$&$\frac{(p^m-1)p^{n}}{(p^{2n}+1)(p^n+1)}$&$\frac{(p^m-1)p^{3n}}{(p^{2n}+1)(p^n+1)}$\\

$p>2$ &$\frac{2(p^m-1)}{(p^{2n}+1)(p^n+1)}$&$\frac{(p^m-1)(p^{2n}-1)}{(p^{2n}+1)(p^n+1)}$&$\frac{(p^m-1)(p^{n}-1)}{(p^{2n}+1)(p^n+1)}$&$\frac{(p^m-1)(p^{3n}+1)}{(p^{2n}+1)(p^n+1)}$ \\
\hline
\end{tabular}
\end{table}
\begin{lemma}\label{lem_VZ gp}
Suppose $G$ is the Suzuki $p$-group which is not a $VZ$-group. Let $\lambda$ be any nontrivial linear character of $Z(G)$ and write $H=\tup{Ker}(\lambda)$. Then the quotient $\overline{G}:=G/H$ is  a $VZ$-group whenever $\overline{G}$ is non-abelian.
\end{lemma}
\begin{proof}
Note that in the cases $G=A_p(m,\theta), B_p(m,\theta,0), D_p(m,\theta,0)$ with $k>2$, we have $G'=Z(G)\cong \bF$ by Table\ref{table1}.
It deduces that $\overline{G}'=Z(G)/H$  and it has size $p$.  Thus for each $g\in \overline{G}\setminus Z(\overline{G})$,   we have $[g,\overline{G}]= \overline{G'}$  and $\overline{G}$ is a $VZ$-group.

As for $G=C_p(m,\theta,0)$ with $k>2$, recall that \[G'=\{(0,0,c):c\in\bF\},\,\, Z(G)=\{(0,b,c):b,c\in\bF\}\cong A_p(m,1).\]
 The claim follows again from the fact that \[|\overline{G}'|=|G'H/H|=\frac{|G'|}{|G'\cap H|}=p.\]
\end{proof}

 We shall use the following lemma.
\begin{lemma}\label{lem_Qx}
Suppose that $r$ is a positive integer. If $r$ is even, take $c\in \bF_{2^r}$ such that $c+c^{2^{r/2}}=1$. Let $Q$ be a function from $\bF_{2^r}$ to $\bF_{2}$ given by
\begin{equation*}
Q(x)=\begin{cases}\sum_{i=0}^{\frac{r-1}{2}}\tup{Tr}_r(x^{2^i+1}),&\tup{ if }r\tup{ is odd},\\\sum_{i=0}^{r/2-1}\tup{Tr}_r(x^{2^i+1})+\tup{Tr}_r(cx^{2^{r/2}+1}),&\tup{ if }r\tup{ is even}.\end{cases}
\end{equation*}
Suppose $U=\{x\in \bF_{2^r}: \tup{Tr}_r(x)=0\}$.
Then for any $x,y\in U$, $Q(x+y)=Q(x)+Q(y)+\tup{Tr}_r(xy)$.
\end{lemma}
\begin{proof}
We only give the details when $r$ is even since the case where $r$ is odd is easier.
 Suppose $r$ is even and let $x,y\in U$. Then
 \begin{small}
\begin{align*}
Q(x+y)+Q(x)+Q(y)&=\sum_{i=0}^{r/2-1}\tup{Tr}_r((x+y)^{2^i+1}-x^{2^i+1}-y^{2^i+1})+\textup{Tr}_r(c(x+y)^{2^{r/2}+1}-cx^{2^{r/2}+1}-cy^{2^{r/2}+1})\\
&=\sum_{i=0}^{r/2-1}\tup{Tr}_r(x^{2^i}y+y^{2^i}x)+\tup{Tr}_r(cx^{2^{r/2}}y+cy^{2^{r/2}}x)\\
&=\sum_{i=0}^{r/2-1}\tup{Tr}_r(x(y^{2^{r-i}}+y^{2^i}))+\tup{Tr}_r(x(c^{2^{r/2}}+c)y^{2^{r/2}})\\
&=\tup{Tr}_r(x)\tup{Tr}_r(y)+\tup{Tr}_r(xy).
\end{align*}
\end{small}
The later equality holds since $c+c^{2^{r/2}}=1$.
The result then follows from the fact that $\tup{Tr}_r(x)=\tup{Tr}_r(y)=0$.
\end{proof}

For $p=2$, let $b_0$ be a fixed element in $\bF$ such that $\tup{Tr}_m(b_0)=1$. Then  for each element $b\in\bF$, we have the  decomposition $b=tb_0+b_1$ with $t\in \bF_2$ and $b_1\in \tup{Ker}(\tup{Tr}_m)$.
Besides, we define an auxiliary function:
\begin{equation}\label{kk}
\kappa:\bF_2\rightarrow \{0,1\}\subset \bZ ,\,\, \overline{0}\mapsto 0,\overline{1}\mapsto 1.
\end{equation}
\begin{lemma}\label{C_Z(G)}
Let $G=C_p(m,\theta,0)$ with $k>2$. Then the linear characters of $Z(G)\cong A_p(m,1)$ are given as follows.
  \begin{itemize}
    \item [(1)]When $p$ is odd,
    $\tup{Lin}(Z(G))=\{\alpha^{(v,w)}:v, w\in \bF\}$, where
    \begin{equation}\label{vartheta^{u,v}}
      \alpha^{(v,w)}(0,b,c)=\psi_w(b)\psi_v(c-\frac{1}{2}b^2).
    \end{equation}
    \item [(2)]When $p=2$, we have
    \[\textup{Lin}(Z(G))=\{\alpha^{(v,w,\epsilon)}: v\in \bF^*, w\in \bF/\bF_2, \epsilon=\pm1\}\cup\{\alpha^{v}: v\in \bF\}.\]
     The expression of  $\alpha^{(v,w,\epsilon)}$ is given by
     \begin{equation}\label{vartheta^{v,w,e}}
       \alpha^{(v,w,\epsilon)}(0,b,c)=(\epsilon i)^{\kappa(t)}(-1)^{Q(b_1)+\tup{Tr}_m(tb_0b_1)}\psi_w(b_1)\psi_v(c).
     \end{equation}
      Here $Q$ is the function from $\bF$ to $\bF_2$ defined in Lemma \ref{lem_Qx}, $i=\sqrt{-1}$ and $t,b_1$ are determined by the decomposition
    $v^{2^{m-1}}b=tb_0+b_1$.
    Besides,
    \begin{equation}\label{eq_alpha}
     \alpha^{v}(0,b,c)=\psi_v(b).
    \end{equation}
  \end{itemize}
  \end{lemma}
 \begin{proof}
 \begin{enumerate}

  \item When $p$ is odd,  the set $K=\{(0,b,\frac{1}{2}b^2):b\in \bF\}$ forms a subgroup of $Z(G)$ which is isomorphic to  $(\bF,+)$, and $Z(G)$ is the direct product of $K$ and the subgroup $0\times 0\times \bF$. Then for each $(0,b,c)\in Z(G)$, we have
  \[
  (0,b,c)=(0,b,\frac{1}{2}b^2)(0, 0, c-\frac{1}{2}b^2).
  \]
   Thus the linear characters of $Z(G)$ are determined by taking tensors of those of $\bF$.

  \item  When $p=2$, for each $v\in \bF^*$, we denote $H_v=\{(0,0,c):\tup{Tr}_m(vc)=0 \}$. Let $K_1=(\bF\times\bF_2,*)$ be the abelian group with multiplication
\begin{equation*}\label{KK}
  (b,x)*(b',x')=(b+b',x+x'+\tup{Tr}_m(bb')).
\end{equation*}
 Then  we have $K_1\cong C_4\times \bZ_2^{m-1}$ via the isomorphism
\begin{equation}\label{Eqn_tata}
\tau:(t b_0+b_1,x)\mapsto (i^{\kappa(t)+2\kappa(x+Q(b_1)+\tup{Tr}_m(t b_0b_1))},b_1),
\end{equation}
where $t\in\bF_2$, $b_1\in\tup{Ker}(\tup{Tr}_m)$.
 Write $\sqrt{v}=v^{2^{m-1}}$. Then $Z(G)/H_v\cong K_1$ via  the isomorphism \[\sigma:\overline{(0,b,c)}\mapsto(\sqrt{v}b, \tup{Tr}_m(vc)).\]
Thus  we deduce that
  \[
  \tup{Lin}(Z(G)/H_v)=\{\alpha^{(v,w,s)}: w\in \bF/ \bF_2, 0\leq s\leq 3\},
  \]
  where \[
 \alpha^{(v,w,s)}\overline{(0,b,c)}=i^{s \kappa(t)}(-1)^{s(Q(b_1)+\tup{Tr}_m(tb_0b_1)+\tup{Tr}_m(vc))}\psi_w(b_1)\]
  with $\sqrt{v}b=tb_0+b_1$, $t\in\bF_2, b_1\in\tup{Ker}(\tup{Tr}_m)$.

On the other hand, for each $v\in\bF$, let $\alpha^v$ be as in \eqref{eq_alpha}. It is clear that $\alpha^v\in\tup{Lin}(Z(G))$.
Those $\alpha^v$'s are all the linear characters containing $H=\{(0,0,c):c\in\bF\}$ in its kernel.
Observe that when $s=0,2$ the kernel of $\alpha^{(v,w,s)}$'s always contain $H$. So we have
\[\{\alpha^{(v,w,s)}: v,w\in\bF,s=0,2\}\subseteq\{\alpha^v:v\in\bF\}.
\]
 Denote $\alpha^{(v,w,s)}$  for $s=1$,$3$ with $\alpha^{(v,w,\epsilon)}$ for $\epsilon=1,-1$, respectively. Then we obtain a set with no repeated elements \[\{\alpha^{(v,w,\epsilon)}: v\in \bF^*, w\in \bF/\bF_2, \epsilon=\pm1\}\cup\{\alpha^{(w,\epsilon)}: w\in \bF/\bF_2, \epsilon=\pm1\}.\]
Since the size of this set is equal to $|Z(G)|$,  the results then follow.
 \end{enumerate}
 \end{proof}
\begin{lemma}\label{G_v}
The description of the $VZ$-group $\overline{G}$ defined in Lemma \ref{lem_VZ gp} and the size of $Z(\overline{G})$ are given as follows:
\begin{enumerate}
  \item Case $G=A_p(m,\theta)$ with $k>2$.  Set $H_v:=\{(0,b):b\in\textup{Ker}(\psi_v)\}$ for each $v\in\bF^*$. Then  $\overline{G}$ can be expressed as $G_v:=G/H_v$. Moreover,
   \begin{enumerate}
     \item when $k$ is odd, $|Z(G_v)|=p^{n+1}$,
     \item when $k$ is even, $|Z(G_v)|=p$ or $p^{2n+1}$ according as $v\in J_1$ or $v\in J_2$.
   \end{enumerate}
  \item  Case $G=B_p(m,\theta,0)$ with $k>2$. Set $H_v=\{(0,0,c):c\in\textup{Ker}(\psi_v)\}$ for each $v\in\bF^*$. Then $\overline{G}$ can be expressed as $G_v:=G/H_v$. Moreover,
     \begin{enumerate}
       \item when $k$ is odd, $|Z(G_v)|=p^{2n+1}$.
       \item  when $k$ is even, $|Z(G_v)|=p$ or $p^{4n+1}$ according as $v\in J_1$ or $v\in J_2$.
     \end{enumerate}
  \item  Case $G=C_p(m,\theta,0)$ with $k>2$. For each $\alpha=\alpha^{(v,w,s)}\in\tup{Lin}(Z(G))$, let $G_{\alpha}=G/\tup{Ker}(\alpha)$, which is the $VZ$-group $\overline{G}$.
   Besides,
    \begin{enumerate}
      \item when $k$ is odd, $|Z(G_\alpha)|=p^{n+1}$ if $p$ is odd and $|Z(G_\alpha)|=2^{n+2}$ if $p=2$.
      \item when $k$ is even,
      \[ \tup{ for }v\in J_1,\,|Z(G_\alpha)|=
      \begin{cases}
        p, &\tup{ if }p\tup{ is odd},\\
        4,&\tup{ if }p=2,
      \end{cases}\]
     \[\tup{ for }v\in J_2,\, |Z(G_\alpha)|=
      \begin{cases}
         p^{2n+1}, &\tup{ if }p\tup{ is odd},\\
        2^{2n+2},&\tup{ if }p=2.
      \end{cases}
      \]
    \end{enumerate}
  \item   Case $G=D_p(m,\theta,0)$ with $k>2$. Set $H_v=\{(0,0,c)\in G:c\in\textup{Ker}(\psi_v)\}$ for each $v\in\bF^*$.  Then  $\overline{G}$ can be expressed as $G_v:=G/H_v$.
    Moveover,
 \begin{enumerate}
   \item when $k$ is odd, $|Z(G_v)|=p^{2n+1}$,
   \item when $k/2$ is odd,   $|Z(G_v)|=p^{2n+1}$ or $p^{4n+1}$ according as $v\in J_1$ or $v\in J_2$,
   \item when $k/2$ is even,
   \[|Z(G_v)|=\left\{
                \begin{array}{ll}
                 p^{6n+1}, &\tup{ if }v\in J_2\cap S_2,\\
                  p^{2n+1}, &\tup{ if }v\in J_2\cap S_1,\\
                  p^{4n+1},&\tup{ if }v\in J_1\cap S_2,\\
                  p,&\tup{ if }v\in J_1\cap S_1.
                \end{array}
              \right.
 \]
 \end{enumerate}
\end{enumerate}

\end{lemma}
\begin{proof}
The expression of $\overline{G}$ is deduced directly from its definition given in Lemma \ref{lem_VZ gp}, so we only calculate the size of $Z(\overline{G})$. The case $G=A_p(m,\theta)$ and $ B_p(m,\theta,0)$ are easier, thus we only consider the remaining two cases.
\begin{enumerate}
  \item  Case $G=C_p(m,\theta,0)$. Note that
 $G_\alpha$ is non-abelian if and only if $G'\nsubseteq \tup{Ker}(\alpha)$, that is, $\alpha\neq \alpha^v, \forall v\in\bF$. So for each $\alpha=\alpha^{(v,w,\epsilon)}\in \tup{Lin}(Z(G))$,
  $G_\alpha$ is non-abelian and is a $VZ$-group. By calculation we have
  \begin{equation}\label{Z(G_vw)}
    Z(G_\alpha)=\{\overline{(a,b,c)}: a=0 \tup { or } \tup{Im}(f_{a,\theta})\subseteq \tup{Ker}(\psi_v), b,c\in \bF \}.
  \end{equation}
 Then from Lemma \ref{C_Z(G)} we obtain that if $p$ is odd and $v\neq 0$, $Z(G)/\tup{Ker}(\alpha^{(v,w)})$ has order $p$. On the contrary, if $p=2$ and $v\neq 0$, $Z(G)/\tup{Ker}(\alpha^{(v,w,\epsilon)})$ has order $4$.
 The size of $Z(G_\alpha)$ then follows from Lemma \ref{lem_JI&J_2} and Lemma \ref{Im(f_a)} (2).
  \item Case $G=D_p(m,\theta,0)$.
  It is deduced that \begin{equation*}\label{Z(D_v)}
    Z(G_v)=\{ \overline{(a,b,c)}:   \tup{Im}(f_{a,\theta})+ \tup{Im}(f_{b,\theta^2})\subseteq\tup{Ker}(\psi_v),c\in \bF \}.
  \end{equation*}
 \begin{enumerate}

  \item When $k$ is odd, we have $\bF_{\theta^4}=\bF_{\theta^2}=\bF_\theta$.
   Then $\tup{Im}(f_{b,\theta^2})$ is still an $\bF_\theta$-hyperplane. Thus for each $b\in\bF^*$, \[\tup{Im}(f_{b,\theta^2})=\tup{Im}(f_{a,\theta})\tup{ for some }a\in\bF^*.\] So we have $|Z(G_v)|=p^{2n+1}$.

  \item When $k/2$ is odd, we have $\bF_{\theta^4}=\bF_{\theta^2}=\bF_{p^{2n}}$. According to Lemma $\ref{lem_JI&J_2}\,(3)$, we deduce  that each $\tup{Ker}(\psi_v)$ contains a unique $\tup{Im}(f_{b,\theta^2})$. \
  Then if $v\in J_1$,
  it follows that 
  \[
  \tup{Im}(f_{a,\theta})\nsubseteq \tup{Ker}(\psi_v) \tup{ for any } a\in \bF^*,\,\tup{and } \tup{Im}(f_{b,\theta^2})\subseteq \tup{Ker}(\psi_v)
\tup{ for some }b\in\bF^*.
  \]
     So  from Lemma \ref{Im(f_a)} (2) we have $|Z(G_v)|=p^{2n+1}$.
      On the other hand, if $v\in J_2$, there exist some $a,b\in\bF^*$ such that
      \[\tup{Im}(f_{a,\theta})\subseteq \tup{Ker}(\psi_v)\tup{ and }\tup{Im}(f_{b,\theta^2})\subseteq \tup{Ker}(\psi_v).\]
  It implies that $|Z(G_v)|=p^{4n+1}$.

   \item When $k/2>1$ is even, we have $\bF_{\theta^4}=\bF_{p^{4n}}$ and $\bF_{\theta^2}=\bF_{p^{2n}}$.  By the definitions of $S_i\cap J_j$, $i,j=1,2$, $(c)$ holds directly.
\end{enumerate}
\end{enumerate}
\end{proof}
\begin{lemma}\label{lem_corre}
\begin{enumerate}
  \item In the case $G=A_p(m,\theta), B_p(m,\theta,0),D_p(m,\theta,0)$ with $k>2$, take $T$ to be a set of coset representatives for $\bF_p^*$ in $\bF^*$.
Then every non-linear irreducible character of $G$ can be determined  by lifting a unique non-linear irreducible character  of $G_v$ for some $v\in T$.
  \item In the case $G=C_p(m,\theta, 0)$ with $k>2$, let $I$ be as in \eqref{eq_I,odd} or \eqref{eq_I,2}. Then every non-linear irreducible character of $G$ can be determined  by lifting a unique non-linear irreducible character  of $G_\alpha$ for some $\alpha\in I$.
\end{enumerate}
\end{lemma}
\begin{proof}
\begin{enumerate}
  \item It  suffices to show that  the kernel of each non-linear irreducible character of $G$ contains exactly a unique $H_v$, $v\in T$, where $H_v$ is as in Lemma \ref{G_v}. We only give the details in the case $G=A_p(m,\theta)$ since in the cases $G=B_p(m,\theta,0)$ and $D_p(m,\theta,0)$ the proof is similar.

       Note that $H_v=H_{v'}$ if and only if $v'=cv$ for some $c\in \bF_p^*$, so we just consider $v\in T$. For $v,v'\in T$, $H_v H_{v'}=Z(G)=G'$ from Table \ref{table_Z(G)&G'}. For each $\chi\in \tup{Irr}(G)$,  $G'\leqslant \tup{Ker}(\chi)$ is equivalent to $\chi\in \tup{Lin}(G)$. So the kernel of each non-linear irreducible character of $G$ contains at most one $H_v$.
  On the other hand,
  we set \[\Delta=\{\chi\in \tup{Irr}_1(G):H_v\leqslant\tup{Ker}(\chi) \tup{ for some } H_v, v\in T\}\subseteq \tup{Irr}_1(G).\]
  If $\Delta=\tup{Irr}_1(G)$, the result then follows.

   Now we claim that $\Delta=\tup{Irr}_1(G)$. We deduce that \[
   |\tup{Irr}_1(G)|=k(G)-|G/G'|=p^n(p^m-1)\]
   from  \eqref{k(G)&IRR(G)} and Table \ref{table_k(G)}.
    On the other hand, by Theorem \ref{main theorem}, there is a correspondence between the set $Y:=\cup_{v\in T}\tup{Irr}_1(G_v)$ and $\Delta$ by associating each element in $Y$ to its lift to $G$. Thus the size of $\Delta$ is equal to that of $Y$.
    
   First we consider the case where $k$ is odd. Since  $G_v$ is a $VZ$-group by Lemma \ref{G_v}, from Theorem \ref{main theorem} and Lemma \ref{G_v} we have
   \[
   |\tup{Irr}_1(G_v)| = |Z(G_v)|-|Z(G_v)/G'_v|= p^{n+1}-p^n= p^n(p-1).
  \]
 Then we conclude that
 \[|\Delta|=|Y|=|T|\cdot |\tup{Irr}_1(G_v)|=\frac{p^m-1}{p-1}p^n(p-1)=|\tup{Irr}_1(G)|,\]
 and the claim holds.
  Similarly, when $k>2$ is even, we deduce from Lemma \ref{G_v} that $|\tup{Irr}_1(G_v)|=p-1$ for each $v\in J_1\cap T$ and $|\tup{Irr}_1(G_v)|=p^{2n}(p-1)$ for each $v\in J_2\cap T$.
  Then
  \begin{small}
  \begin{equation*}
    |\Delta|=|J_1\cap T|(p-1)+|J_2\cap T|p^{2n}(p-1)=\frac{|J_1|}{p-1}(p-1)+\frac{|J_2|}{p-1}p^{2n}(p-1)=|\tup{Irr}_1(G)|,
  \end{equation*}
 \end{small}
  and the claim holds again.
  \item In the case $G=C_p(m,\theta, 0)$, observe  that if $p>2$, $\tup{Ker}(\alpha^{(v,w)})=\tup{Ker}(\alpha^{(v',w')})$ if and only if $v'=sv$ and $w'=sw$ for some $s\in \bF_p^*$. As for  $p=2$,
we have $\tup{Ker}(\alpha^{(v,w,1)})=\tup{Ker}(\alpha^{(v,w,-1)})$. So we consider $\alpha\in I$, where $I$ is given by \eqref{eq_I,odd} or \eqref{eq_I,2}.
The remaining of the proof follows from the same method used in (1).

%

\end{enumerate}
\end{proof}
\begin{remark}\label{sizeZ(A_v)}
\begin{enumerate}
  \item In the case $G=A_p(m,\theta)$,  \cite[Lemma 2.7, (ii)]{Suzuki} failed as $\overline{G}/\tup{ker}(\eta)$ is not extra-special any more  when $p=2$. Hence there is a small error in the non-linear irreducible characters of $A_2(m,\theta)$ with $o(\theta)=k$ odd  listed in \cite[Theorem 2.3 (i)]{Suzuki}.
  \item  In the case $G=C_p(m,\theta,0)$, when $k>1$ is odd and $p=2$, $G_\alpha$ is not isomorphic to $A_p(m,\theta)/H_v$ and hence the proof of  \cite[Theorem 3.2(iii)]{Suzuki} is not true.
\end{enumerate}

\end{remark}
Theorem \ref{thm 1} now follows from Lemma \ref{lem_corre} and Theorem \ref{main theorem} and we have proved our main results.

\section{The expressions of irreducible characters of $G=A_p(m,\theta)$}\label{Sec5}
To give the explicit expressions of the irreducible characters of the Suzuki $p$-group $G$, it remains to determine $\tup{Lin}(G/G')$ and $\tup{Lin}(Z(\overline{G}))$.
In this section, we set $G=A_p(m,\theta)$ and  give $\tup{Lin}(G/G')$ with $k=2$ and $\tup{Lin}(Z(\overline{G}))$ with $k>2$.  As for the remaining three cases, those  can also be obtained by the similar method used in this section. To make this paper more readable and breviate, we do not give the details of them.

\subsection{The case $k=2$}\label{k=2of A}\quad

\begin{prop}\label{prop_Lin(A/A')}
Let $G=A_p(m,\theta)$ with $k=2$. Then
 $\tup{Lin}(G/G')=\{\chi_1^{(v,w)}:v\in\bF,w\in\bF_{p^n}\}$, where
 \begin{equation}\label{eq_Lin(A/A')}
 \chi_1^{(v,w)}\overline{(a,b)}=\psi_v(a)\phi_w(b+b^{p^n}-a^{p^n+1}).
 \end{equation}
Here $\psi_v$ and $\phi_w$ are defined in \eqref{psi_v} and \eqref{phi_w}, respectively.
\end{prop}
\begin{proof}
In this case, $a^{\theta}=a^{p^n}$ for  each $a\in \bF$.
Set  $K:=(\bF,+)\times (\bF_{p^n},+)$, then it is  an abelian group of order $p^{m+n}$.
Define \begin{equation*}
\rho: G/G'\rightarrow K,\, (\overline{a,b})\mapsto (a,-aa^\theta +b+b^{\theta}).
\end{equation*}
We claim that $\rho$ is an isomorphism.
Recall that $G'=\{(0,x):x\in \tup{Im}(f_{1,\theta}) \}$ by Table \ref{table_Z(G)&G'} and  $\tup{Im}(f_{1,\theta})=\{x\in\bF:x^\theta+x=0\}$ by Lemma \ref{Im(f_a)} (1). So we conclude that
 $(\overline{a,b})=(\overline{a,b'})$ if and only if $b-b'\in \tup{Im}(f_{1,\theta})$.  That is to say,
 \[
 (\overline{a,b})=(\overline{a,b'})\tup{ if and only if }(a,-aa^\theta +b+b^{\theta})=(a,-aa^\theta +b'+b'^{\theta}).
 \]
 Thus the map $\rho$ is well defined and injective. Since $|G/G'|=|K|$, we have $\rho$ is an bijection.
By direct calculation,
we obtain that $\rho$ is an isomorphism.
Then the linear characters of $K$ are of the form
\[
\psi_v\otimes\phi_w(x,y)=\psi_v(x)\phi_w(y),
\]
so the linear characters of $G/G'$  follow from the isomorphism $\rho$.

\end{proof}
\begin{remark}
  There is a mistake about the linear characters of $G$ in \cite[Theorem 2.3 (ii)]{Suzuki}, and the correct version can be deduced  from the above proposition directly.
\end{remark}

\subsection{The case $k>2$ }\label{kodd of A}\quad

\begin{lemma}\label{lem_kerpsiv1}
  Suppose that either $k$ is odd and $p=2$, or $k$ is even. If $v\in(a^{\theta+1})^{-1}\bF_{\theta}^*$, then there is  $a_v\in\bF^*$ such  that $va_v^{\theta+1}=1$.
\end{lemma}
\begin{proof}
Write $b:=va^{\theta+1}$ and suppose that $b\in\bF_\theta^*$.
We need to show that there is $y\in\bF^*$ such that $y^{\theta+1}=b$. So that $a_v:=ay^{-1}$ is a desired element.
It suffices to show
\[
\{y^{\theta+1}: y\in\bF^*\}\supseteq\bF_{p^n}^*.
\]
 Observe that the left hand side is a multiplicative subgroup of $\bF^*$ of order $\frac{p^m-1}{\gcd(p^l+1,p^m-1)}$, and the right hand side is a subgroup of order $p^n-1$. Write $d=\gcd(p^l+1,p^m-1)$.
Then it suffices to show
\[
p^n-1 \mid \frac{p^m-1}{d}, \, i.e. \,\,  d(p^n-1)\mid p^m-1.
\]
From Lemma \ref{lem(p^l+1,p^m-1)}, $d(p^n-1)\mid p^m-1$ if and only if  $k$ is odd and $p=2$, or $k$ is even, and the proof is completed.

\end{proof}

\subsubsection{ The case $k>2$, $k$ odd}\qquad


  In the case $p=2$, we fix an element $u_0\in\bF_{2^n}$ such that $\tup{Tr}_{n}(u_0)=1$ and define
  \[U=\{x\in \bF_{2^n}: \tup{Tr}_n(x)=0\}.
   \]
   Then $\bF_{2^n}=U\oplus\bF_2 u_0$.  For an element $v\in\bF^*$, let $H_v$ be as in Lemma \ref{G_v} and set $G_v=G/H_v$.
  By Lemma \ref{lem_kerpsiv1}, there exists an element $a_v\in\bF^*$ such that $va_v^{\theta+1}=1$.
 Then  we conclude that
 \begin{equation}\label{Z(G_v)of k odd}
   Z(G_v)=\{\overline{(a_v(\delta u_0+u_1),b)}: \delta\in \bF_2, u_1\in U,b\in\bF\}.
 \end{equation}

  In the case that $p$ is odd, take $\gamma$ to be a primitive element of $\bF^*$.  Then $\langle \gamma^{p^l+1}\rangle$  is a subgroup of $ \bF^*$ of order
  \[(p^m-1)/\gcd(p^l+1,p^m-1)=(p^m-1)/2\]
   by Lemma \ref{lem(p^l+1,p^m-1)}, so it is the set of squares of $\bF^*$.
  For each $v\in \bF^*$, let $H_v$ be as in Lemma \ref{G_v} and set $G_v=G/H_v$.

  \begin{lemma}\label{lem_xv}
    For each $v\in \bF^*$, the set  $\{va^{\theta+1}:a\in\bF^*,\tup{Im}(f_{a,\theta})\subseteq \tup{Ker}(\psi_v)\}$ is the set of squares or non-squares of $\bF_{p^n}^*$, according as $v$ is  a square  of $\bF^*$ or not.
  \end{lemma}
   \begin{proof}
   Fix an $a_0\in\bF^*$ such that $\tup{Im}(f_{a_0,\theta})\subseteq \tup{Ker}(\psi_v)$. We have
     \[
    S:= \{va^{\theta+1}:a\in\bF^*,\tup{Im}(f_{a,\theta})\subseteq \tup{Ker}(\psi_v)\}= \{v(a_0u)^{\theta+1}:u\in\bF_{p^n}^*\}
     =va_0^{\theta+1}\{u^2:u\in \bF_\theta^*\},
     \]
     where the second equality is deduced from Lemma $\ref{Im(f_a)} \,(2)$.
     From Lemma \ref{lem_fainv}, we know that $va_0^{\theta+1}\in\bF_\theta^*$.
     So $S$ is the set of squares or non-squares of $\bF_{p^n}^*$, according as $va_0^{\theta+1}$ is a square of $\bF_{p^n}^*$ or not.
    Thus  $va_0^{\theta+1}$ is a square of $\bF_{p^n}^*$ if and only if
     $va^{\theta+1}=1$  for some $a\in\bF^*$.
      That is  $v=(a^{-1})^{\theta+1}\in\langle \gamma^{p^l+1}\rangle$, the set of squares of $\bF^*$.
    So the claim is proved.
   \end{proof}
   Let $x_0$ be a fixed non-square in $\bF_{p^n}^*$.  From Lemma \ref{lem_xv}, there is $a_v$ such that $va_v^{\theta+1}=1$ if $v$ is a square of $\bF^*$. If $v$ is a non-square of $\bF^*$, there is $a_v$ such that $va_v^{\theta+1}=x_0$.
Then we express $Z(G_v)$ as follows:
  \begin{equation}\label{Z(G_v)}
    Z(G_v)=\{ \overline{(a_vu,b)}: u\in\bF_{p^n}, b\in \bF \}.
  \end{equation}
  Besides, set
  \begin{equation}\label{eq_xv}
    x_v=
  \begin{cases}
    \,1, &\tup{ if } v \tup{ is a square of }\bF^*,\\
   \, x_0, &\tup{ if } v \tup{ is a non-square of }\bF^*.
  \end{cases}
  \end{equation}
\begin{lemma}\label{lem_ZGB2}
Take notation as above and assume that $k>2$ is odd. Let $\psi_v,\phi_w$ be as in \eqref{psi_v} and \eqref{phi_w}, respectively. Then the following holds:
\begin{itemize}
  \item[(1)] When $p=2$,
$\tup{Lin}(Z(G_v))=\{\lambda^{(v,w,s)}:w\in\bF_{2^n}/\bF_2,  0\leq s\leq3\}$, where
\begin{equation}\label{IrrZ(G_v)1}
  \lambda^{(v,w,s)}:(\overline{a_v(u_1+\delta u_0),b})\longmapsto {i}^{s\kappa(\delta)}(-1)^{s(Q(u_1)+\tup{Tr}_n(\delta u_0u_1))}\phi_w(u_1)\psi_{sv}(b).
\end{equation}
Here  
 $\kappa$ is defined in \eqref{kk} and $Q$ is the function from $\bF_{2^n}$ to $\bF_2$ defined in Lemma \ref{lem_Qx}.
\item[(2)] When $p$ is odd,
$\tup{Lin}(Z(G_v))=\{\lambda^{(v,w,s)}:w\in\bF_{p^n},  0\leq s\leq p-1\}$ where
\begin{equation}\label{IrrZ(G_v)2}
  \lambda^{(v,w,s)}:(\overline{a_vu,b})\longmapsto\xi_p^{-\frac{1}{2}s\tup{Tr}_m(x_vu^2)}\phi_w(u)\psi_{sv}(b).
\end{equation}
\end{itemize}

\begin{proof}
Let $K=(\bF_{p^n}\times\bF_p,*)$ be the group with multiplication $*$ defined by
\begin{equation*}\label{K}
  (u,x)*(u',x')=(u+u',x+x'+\tup{Tr}_m(va_v^{\theta+1} uu')).
\end{equation*}
Then $Z(G_v)$ is isomorphic to  $K$ under the isomorphism:
\begin{equation}\label{Eqn_sigma}
\sigma:(\overline{a_vu,b})\mapsto (u,\tup{Tr}_m(vb)), \forall \,(\overline{a_vu,b})\in Z(G_v).
\end{equation}
\begin{enumerate}

\item When $p=2$,  we have $va_v^{\theta+1}=1$.
Let
 $ N=\{(u_1,Q(u_1)):u_1\in U\}$.
 Then the set $N$ forms a subgroup of $K$, since $Q(u_1+u_1')=Q(u_1)+Q(u_1')+\Tr_n(u_1u_1')$ for any $u_1,u_1'\in U$ and
 \[
 \tup{Tr}_m(uu')=\tup{Tr}_n(\tup{Tr}_{m/n}(uu'))=\tup{Tr}_n(kuu')=\tup{Tr}_n(uu') \tup{ for any }u,u'\in\bF_{2^n}.
 \]
 For any $(u,x)=(\delta u_0+u_1,x)\in K$, we have the decomposition
\begin{equation*}\label{decom of K}
 (u,x)=(u_1,Q(u_1))*(\delta u_0,x+Q(u_1)+\tup{Tr}_n(\delta u_0u_1)),
\end{equation*}
where $(u_1,Q(u_1))\in N$, and
\[ (\delta u_0,x+Q(u_1)+\tup{Tr}_n(\delta u_0u_1))=(u_0,0)^{\kappa(\delta)+2(x+Q(u_1)+\tup{Tr}_n(\delta u_0u_1))}\in \langle(u_0,0)\rangle.\]
It implies that $ K=N\times \langle(u_0,0)\rangle$. So its characters are determined by taking the tensors of those of $N$ and  $\langle(u_0,0)\rangle$.
 Therefore, the irreducible characters of $Z(G_v)$ follow from the isomorphism $\sigma$ given in \eqref{Eqn_sigma}.
 Here we observe that for  $\phi_w,\phi_{w'}\in\tup{Irr}(\bF_{2^n},+)$,  $\phi_{w}|_U=\phi_{w'}|_U$ if and only if $w\in w'+\bF_2$.
Thus we chose $w\in\bF_{2^n}/\bF_2$.

 \item When $p$ is odd,
Let
 \[N=\{(u,\frac{1}{2}\Tr_m(x_vu^2)):u\in \bF_{p^n}\},\]
which is a subgroup of $K$. Besides, we have $K=N\times (0\times \bF_p)$ under the decomposition

\[(u,x)=(u,\frac{1}{2}\Tr_m(x_vu^2))*(0,x-\frac{1}{2}\Tr_m(x_vu^2)),\] where $(u,x)\in K$.
    Similarly, the claim follows again via the isomorphism $\sigma$ given in \eqref{Eqn_sigma}.
\end{enumerate}
\end{proof}
\end{lemma}

\subsubsection{ The case $k>2$, $k$ even}\quad

In this case, $\theta^2$ has order $k/2$, and $\bF_{\theta^2}=\bF_{p^{2n}}$, where recall that $\bF_\theta=\bF_{p^n}$.
Choose $j\in \bF_{p^{2n}}$ such that $j+j^{p^n}=1$. Then $\bF_{p^{2n}}=\bF_{p^n}\oplus j\bF_{p^n}$. Let $J_1, J_2$ be as in Lemma \ref{lem_JI&J_2}.
For $v\in J_1$, we have $ Z(G_v)=G'_v$ has order $p$ by Lemma \ref{G_v}. For $v\in J_2$, there exists $a_v\in\bF^*$ such that $va_v^{\theta+1}=1$ by  Lemma \ref{lem_kerpsiv1} and Lemma \ref{lem_fainv},
 so we have \[
 Z(G_v)=\{\overline{(a_vu,b)}:u\in \bF_{p^{2n}},b\in\bF\}.
 \]

Similar to Lemma \ref{lem_ZGB2}, we get the following lemma.
\begin{lemma}\label{lem_ZGB3}
Let $G=A_p(m,\theta)$ with $o(\theta)=k>2$ even and take notations as above. For each $v\in \bF^*$, let $H_v$ be as in Lemma \ref{G_v} and set $G_v:=G/H_v$. Then the following holds:
\begin{itemize}
  \item[(1)] For $v\in J_1$, we have
\[\tup{Lin}(Z(G_v))=\{\lambda^{(v,s)}:  0\leq s\leq p-1\},\] where
 $ \lambda^{(v,s)}(\overline{0,b})= \psi_{sv}(b).$
\item[(2)]For $v\in J_2$, 
\begin{itemize}
  \item[(i)] if $p=2$,
$\tup{Lin}(Z(G_v))=\{\lambda^{(v,w_1,w_2,s)}:w_1,w_2\in\bF_{p^n},  s=0,1\}$, where
\begin{equation}\label{IrrZ(G_v)222}
  \lambda^{(v,w_1,w_2,s)}:(\overline{a_v(u_1+ju_2),b})\longmapsto\phi_{w_1}(u_1)\phi_{w_2}(u_2)(-1)^{s(\tup{Tr}_m(vb)+\tup{Tr}_m(ju_1u_2))},
  \end{equation}
  \item[(ii)]if $p$ is odd,
$\tup{Lin}(Z(G_v))=\{\lambda^{(v,w,s)}:w\in\bF_{p^{2n}},  0\leq s\leq p-1\}$, where
\begin{equation}\label{IrrZ(G_v)22}
  \lambda^{(v,w,s)}:(\overline{a_vu,b})\longmapsto\phi_{w}(u)\xi_p^{s(\tup{Tr}_m(vb)-\frac{1}{2}\tup{Tr}_m(u^{\theta+1}))}. \end{equation}
\end{itemize}
\end{itemize}
\end{lemma}

\begin{proof}
Here we only give the proof of $v\in J_2$, and the case $v\in J_1$ is easier and omitted.
Let $K=(\bF_{p^{2n}}\times\bF_p,*)$ be the group with multiplication $*$ defined by
\begin{equation*}
  (u,x)*(u',x')=(u+u',x+x'+\tup{Tr}_m(uu'^\theta)).
\end{equation*}
 Then $Z(G_v)$ is isomorphic to  $K$ under the isomorphism:
\begin{equation*}
\sigma:(\overline{a_vu,b})\mapsto (u,\tup{Tr}_m(vb)), \tup{ for all }(\overline{a_vu,b})\in Z(G_v).
\end{equation*}
\begin{enumerate}

 \item When $p=2$, for $x,x'\in\bF_{2^n}$, we have
\[
\Tr_m(xx'^\theta)=\Tr_n(\Tr_{m/n}(xx'))=k\Tr_n(xx')=0,\]
and
\[\Tr_m(jx (jx')^\theta)=\Tr_m(j^{\theta+1}x x')=\Tr_n(\Tr_{m/n}(j^{\theta+1}xx'))=k\Tr_n(j^{\theta+1}xx')=0.
\]
  Thus $N_1:=\{(u_1,0):u_1\in \bF_{2^n}\}$ and $N_2:=\{(ju_2,0):u_2\in \bF_{2^n}\}$ are  subgroups of $K$. Set $N_3=:0\times\bF_p$. Then $K$ is the direct product of $N_1,N_2,N_3$ and so its characters are determined by taking the tensors of those of  $N_1,N_2,N_3$.
  For $u_1,u_2\in\bF_{2^n}$, set $x:=u_1(ju_2)^\theta$. We have
  \[
  \Tr_m(u_1(ju_2)^\theta)+ \Tr_m(u_1 ju_2)=\Tr_m(x+x^\theta)=\Tr_n(\Tr_{m/n}(x+x^\theta))=k\Tr_n(x+x^\theta)=0.
  \]
  Thus $\Tr_m(u_1(ju_2)^\theta)= \Tr_m(u_1ju_2)$.
   Then for $(u,x)\in K$, we have the following decomposition:
 \[
 (u,x)=(u_1+ju_2,x)=(u_1,0)*(ju_2,0)*(0,x+\Tr_m(ju_1u_2)).
 \]
Therefore the claim follows.

 \item
When $p$ is odd,  let $N=\{(u,\frac{1}{2}\Tr_m(u^{\theta+1})): u\in \bF_{p^{2n}}\}$. Then $K=N\times (0\times \bF_p)$ and the decomposition is as follows:
\[
(u,x)=(u,\frac{1}{2}\Tr_m(u^{\theta+1}))*(0, x-\frac{1}{2}\Tr_m(u^{\theta+1})).\]
So the claim holds.
\end{enumerate}
\end{proof}

\end{document}